\documentclass[11pt, a4paper]{article}

\usepackage[utf8]{inputenc}
\usepackage[T1]{fontenc}
\usepackage{geometry}
\usepackage{amsmath, amssymb, amsthm} 
\usepackage{graphicx} 
\usepackage{hyperref} 
\usepackage{xcolor} 
\usepackage{enumitem} 
\usepackage{booktabs} 
\usepackage{comment}

\usepackage[backend=bibtex, style=numeric, sorting=nyt]{biblatex}
\title{Pseudofiniteness of the Farey Graph}
\author{Connor M. Lockhart \\ \texttt{connorl@umd.edu} \\ \textit{University of Maryland}}
\date{\today}

\newtheorem{theorem}{Theorem}
\newtheorem{lemma}[theorem]{Lemma}
\newtheorem{proposition}[theorem]{Proposition}

\theoremstyle{definition}
\newtheorem{definition}{Definition}

\newtheorem{remark}{Remark}

\begin{document}

\maketitle

\begin{abstract}
We prove that the theory of the Farey graph is pseudofinite by constructing a sequence of finite structures that satisfy increasingly large subsets of its first-order axiomatization.
This graph is an important object in the study of curve graphs, and its model-theoretic properties have been explored in the broader context of curve graphs of surfaces in \cite{disarlo2023modeltheorycurvegraph}.
The theory of the Farey graph was recently axiomatized by Tent and Mohammadi in \cite{TentMohammadi2025}.
We show that while no finite planar graph can satisfy these axioms for sufficiently large substructures, they can be satisfied by triangulations densely embedded on orientable surfaces of higher genus.
By applying a result of Archdeacon, Hartsfield, and Little \cite{Archdeacon1996} on the existence of triangulations with representativity and connectedness, we establish that every finite subset of the theory of the Farey graph has a finite model as desired.
\end{abstract}

\section{Introduction}

The primary objective of this paper is to establish the pseudofiniteness of the Farey graph.
The Farey graph is an important object appearing in various branches of mathematics, in particular number theory and hyperbolic geometry. 
Additionally, the Farey graph can be identified with the curve graph of the punctured torus, which has been studied from a model-theoretic perspective in \cite{disarlo2023modeltheorycurvegraph}.
We utilize the work of Tent and Mohammadi in \cite{TentMohammadi2025} and some results in topological graph theory to explore when finite subsets of this theory can be realized by finite models. Our main result is to show that the theory of the Farey graph is pseudofinite. 

\begin{definition}[Pseudofinite Structure]
    An infinite first-order $\mathcal{L}$-structure $\mathcal{M}$ is \textit{pseudofinite} if for every $\mathcal{L}$-sentence $\phi$, $M\vDash \phi$ implies that there is a finite $\mathcal{L}$-structure $\mathcal{M}_0$ such that $\mathcal{M}_0\vDash \phi$.   
\end{definition}

\begin{definition}[Pseudofinite Theory]
    A complete first-order theory $T$ is \textit{pseudofinite} if $T$ is the theory of some pseudofinite structure.
\end{definition}

\subsection{The Farey Graph}

We first recall its classical construction of the Farey Graph on the set of rational numbers.

\begin{definition}[Classical Farey Graph]
    The \textit{Farey graph} $\mathcal{F}$ is the graph with vertex set $V = \mathbb{Q} \cup \{\infty\}$.
    Representing each vertex as a reduced fraction $p/q$ where $q \ge 0$ and $\gcd(p, q) = 1$ (with $\infty$ represented as $1/0$).
    Two vertices $p/q$ and $r/s$ are adjacent if and only if
    \[ |ps - qr| = 1. \]
\end{definition}

Alternatively, the Farey graph can be defined through an iterative combinatorial process starting from a single edge.

\begin{definition}[Combinatorial Construction]
    The Farey graph can be viewed as an increasing union of graphs $\mathcal{F}_0 \subset \mathcal{F}_1 \subset \mathcal{F}_2 \dots$, where:
    \begin{itemize}
        \item $\mathcal{F}_0$ consists of the single edge between $0/1$ and $1/0$.
        \item $\mathcal{F}_{n+1}$ is obtained from $\mathcal{F}_n$ by adding a new vertex for every edge $\{p/q, r/s\}$ in $\mathcal{F}_n$ present at stage $n$. The new vertex is the mediant $(p+r)/(q+s)$, which is connected to both $p/q$ and $r/s$, thereby forming a new triangle.
    \end{itemize}
    The full Farey graph is the limit of this process, $\mathcal{F} = \bigcup_{n=0}^\infty \mathcal{F}_n$.
\end{definition}

\subsection{Graph Theoretic Preliminaries}
Throughout this paper, all graphs are assumed to be \textit{simple}, meaning they contain no loops or multiple edges.
Formally, a graph $G = (V, E)$ consists of a set of vertices $V$ and a set of edges $E \subseteq [V]^2$, where $[V]^2$ denotes the set of 2-element subsets of $V$.
Equivalently, the edge relation $E$ can be viewed as a symmetric and irreflexive binary relation on $V$.

We recall the definition of vertex connectivity, which measures the structural robustness of a graph.

\begin{definition}[Vertex Connectivity]
    The \textit{vertex connectivity} of a graph $G$, denoted by $\kappa(G)$, is the minimum number of vertices whose removal results in a disconnected graph or a single vertex.
    A graph is said to be \textit{$k$-connected} if $\kappa(G) \ge k$.
\end{definition}
\begin{definition}
For a vertex $v \in V$, the \textit{neighborhood} of $v$, denoted $N[v]$, is the set $\{v\} \cup \{u \in V : (u, v) \in E\}$.
That is, $N[v]$ is the closed neighborhood of $v$.
\end{definition}

We now establish some basic combinatorial definitions required to discuss the structural properties of the Farey graph. The following definition comes from \cite{TentMohammadi2025}. 

\begin{definition}
Let $G = (V, E)$ be a graph.
A vertex $v \in V$ is \textit{removable} if it satisfies one of the following conditions:
\begin{enumerate}
\item $\deg(v) \leq 1$.
\item $\deg(v) = 2$ and the neighbors of $v$ are adjacent.
\end{enumerate}
\end{definition}

The property of being removable can be conveniently rephrased in terms of the clique size of the vertex's neighborhood.

\begin{proposition}
\label{prop:removable_equiv}
A vertex is removable if and only if its neighborhood $N[v]$ is a clique of size at most 3.
In particular, if $G$ is $K_4$-free, a vertex is removable if and only if its neighborhood is a clique.
\end{proposition}

\begin{proof}
Let $v$ be a vertex.
If $v$ is removable, then $\deg(v) \le 1$ or ($\deg(v)=2$ and neighbors adjacent).
Since $v$ is adjacent to all its neighbors, $N[v]$ induces a clique of size $\deg(v)+1 \le 3$.
Conversely, if $N[v]$ is a clique of size $\le 3$, then $\deg(v) \le 2$ and neighbors are adjacent if $\deg(v)=2$ (since $N[v] \setminus \{v\}$ is a clique), so $v$ is removable.
If $G$ is $K_4$-free, any clique has size at most 3.
Thus for any vertex $v$, if $N[v]$ is a clique, $|N[v]| \le 3$ (otherwise $N[v]$ would be a clique of size at least 4, which is $K_4$).
Thus, in a $K_4$-free graph, $v$ is removable if and only if $N[v]$ is a clique.
\end{proof}

This property is particularly relevant for chordal graphs, which are known to always possess such vertices.

\begin{proposition}
\label{finiteK4freechordal}
Any finite, $K_4$-free, chordal graph has a removable vertex.
\end{proposition}

\begin{proof}
Via the perfect elimination ordering, every finite chordal graph has a vertex whose neighborhood $N[v]$ is a clique.
In a $K_4$-free graph, a vertex whose neighborhood is a clique is removable by Proposition \ref{prop:removable_equiv}.
\end{proof}

Using these combinatorial foundations, we can introduce the axiomatization of the Farey graph essentially  developed by Tent and Mohammadi.\footnote{In the version of \cite{TentMohammadi2025} on the ArXiv as of March 24th 2026, the axioms presented neglect to include that every vertex has degree at least 1, which leads to a non complete theory, As the Farey graph union a single isolated point would satisfy the axioms presented but satisfy the sentence $\phi$, $\exists x\forall y: \neg E(x,y)$ which is not satisfied by the Farey graph with no additional points.}
Their result provides a first-order characterization of the Farey graph based on local behavior of triangles and the consistent existence of removable vertices.

\begin{theorem}[\cite{TentMohammadi2025}]
    Let $n\in \mathbb{N}$ be nonzero.
The theory of the Farey Graph is axiomatized by:
    \begin{enumerate}
        \item $\phi_0:=$ every edge is contained in  exactly two triangles and every vertex has degree at least 3,
        \item $\psi_{n}:=$ every substructure of size at most $n$ has a removable vertex. 
    \end{enumerate}
\end{theorem}

It is worth noting that the axioms $\psi_n$ form a hierarchy of increasingly strict conditions.

\begin{remark}
    Let $n\geq 1$, $\psi_{n+1}$ implies $\psi_{n}$. 
\end{remark}

We first observe that these axioms are quite restrictive for finite planar graphs.
Via the usual embedding into the disk, the Farey graph is planar. However, as the following proposition shows, no non-empty finite graph (and hence no finite substructure) can satisfy the axiomatization $\phi_0$ and $\psi_n$.
\begin{proposition}
For $n \geq 6$, no finite planar graph satisfies the axioms $\phi_0$ and $\psi_n$.
\end{proposition}
\begin{proof}
Let $G$ be a finite planar graph satisfying $\phi_0$. 
It is a classical graph theory result that every finite planar graph contains a vertex $v$ of degree at most $5$ and so fix $v_0\in G$ with $\deg(v_0)\leq 5$. 

Consider the substructure $N[v_0]\subseteq G$. 
From $\phi_0$, we know that every vertex is of degree at least 3 considered as an element of $G$. 
Without loss of generality let $w\neq v$ be a vertex in $N[v_0]$ and $z_1,z_2$ be the two vertices connected to both $w$ and $v$ forming triangles. 
Since $z_1,z_2$ are distinct and also contained in $N[v_0]$ then $w$ has degree at least $3$ considered as an element of $N[v_0]$ and cannot be removable.
Thus, $N[v_0]$ has cardinality $|N[v_0]|\leq 6$ and no removable elements, contradicting $\psi_6$ as desired. 
\end{proof}

Despite these limitations for planar graphs, the pseudofiniteness of the Farey graph can be established by considering graphs embedded on higher-genus surfaces.
We rely on the existence of triangulations that are sufficiently highly connected and locally planar in large neighborhoods to satisfy $\psi_n$ and $\phi_0$.

In this paper, we will show that the theory of the Farey graph, $Th(F)$, satisfies this condition by constructing a sequence of finite graphs that satisfy increasingly large finite subsets of its axiomatization.

\begin{proposition}[Existence of a Densely Embedded Triangulation]
    \label{dense embedding}
    For each $n$, there exists a simple graph $G_n$ and an orientable surface $\Sigma$ where $G_n$ is an $n$-connected triangulation of $\Sigma$ with representativity $n$. 
\end{proposition}

\begin{proof}
    Theorem 1.1 of \cite{Archdeacon1996}.
\end{proof}

It should be noted that the above theorem 1.1 in \cite{Archdeacon1996} is actually stronger than stated in \autoref{dense embedding} and contains more conditions than we need. 
\section{Graph Embeddings}
To analyze the properties of the graphs $G_n$, we must precisely define how they are situated within their associated surfaces.
We begin with the general definition of a graph embedding.

\begin{definition}[Graph Embeddings]
    An \textit{embedding} of a graph $G$ on a surface $\Sigma$ is a map of $G$ into $\Sigma$ where vertices of $G$ are points on $\Sigma$, and edges between vertices are homeomorphisms of $[0,1]$ such that the interior of the images of any two edges are disjoint and do not intersect the images of any vertices besides the appropriate end points.
    We will refer to $G$ as both the graph and its image in $\Sigma$. 
\end{definition}

An embedding naturally divides the surface into several components.

\begin{definition}[Face]
    Given an embedding of a graph $G$ into a surface $\Sigma$, a \textit{face} $F$ of $G$ is a connected component of $\Sigma\setminus G$.
The \textit{boundary of a face}  $\partial F$ is the topological boundary of $F$ inside of $\Sigma$.  
\end{definition}

In this paper, we focus exclusively on embeddings where every face is a triangle.

\begin{definition}[Triangulation]
    A graph $G$ embedded into a surface $\Sigma$ is a \textit{triangulation} if the boundary of every face contains $3$ distinct vertices and $3$ distinct edges. 
\end{definition}

Finally, we introduce the concept of representativity, which measures the smallest number of intersections between the graph and a non-contractible cycle on the surface.

\begin{definition}[Representativity]
    Let $G$ be a graph embedded on a surface $\Sigma$.
Find a non-contractible cycle $C$ in $\Sigma$ for which $n=|C\cap G|$ is minimized.
We call $n$ the \textit{representativity} of the embedding and say that the embedding is $n$-representative.
(This is also known as the \textit{facial width} in some sources.)
\end{definition}

\section{Main Results}
We now prove that the graphs $G_n$ provided by \ref{dense embedding} satisfy the required axioms for large $n$.
Our strategy involves showing that small induced subgraphs of $G_n$ behave like planar graphs due to the high representativity of the embedding.

\begin{lemma}
\label{lemma:planarity}
Any induced subgraph $H$ of $G_n$ with $|V(H)| < n$ is planar.
\end{lemma}

\begin{proof}
We employ the universal covering space of the surface $\Sigma$ to establish planarity.
Let $\tilde{\Sigma}$ be the universal cover of $\Sigma$, and let $\pi: \tilde{\Sigma} \to \Sigma$ be the associated covering map.
Since $\Sigma$ is a surface of genus $g \geq 1$, $\tilde{\Sigma}$ is homeomorphic to the plane $\mathbb{R}^2$.
Let $\tilde{G}_n = \pi^{-1}(G_n)$.
Since $\tilde{G}_n$ is embedded in $\tilde{\Sigma} \cong \mathbb{R}^2$, it is an infinite planar graph.

Consider the subgraph $H \subset G_n$.
Since the disjoint union of planar graphs is planar, we may assume $H$ is connected.
Let $v \in V(H)$ be a vertex and let $\tilde{v} \in V(\tilde{G}_n)$ be a lift of $v$ such that $\pi(\tilde{v}) = v$.
Let $\tilde{H}$ be the connected component of $\pi^{-1}(H)$ containing $\tilde{v}$.
The restriction of the projection $\pi|_{\tilde{H}}: \tilde{H} \to H$ is a locally injective graph homomorphism.
To prove that $H$ is planar, it suffices to show that $\pi|_{\tilde{H}}$ is a graph isomorphism, which implies $H$ is isomorphic to the planar subgraph $\tilde{H}$.

Suppose for the sake of contradiction that $\pi|_{\tilde{H}}$ is not injective.
Then there exist distinct vertices $\tilde{x}, \tilde{y} \in V(\tilde{H})$ such that $\pi(\tilde{x}) = \pi(\tilde{y}) = z \in V(H)$.
Since $\tilde{H}$ is connected, there exists a path $\tilde{P}$ in $\tilde{H}$ connecting $\tilde{x}$ and $\tilde{y}$.
The projection $P = \pi(\tilde{P})$ is a closed walk in $H$ rooted at $z$.
Since $\tilde{x}$ and $\tilde{y}$ are distinct lifts of the same vertex, the walk $P$ represents a non-trivial element of the fundamental group $\pi_1(\Sigma, z)$.
Thus, $P$ contains a non-contractible cycle.

The length of this non-contractible cycle is bounded by the number of vertices in $H$, so length$(P) \leq |V(H)| < n$.
However, the representativity of $G_n$ is defined as the minimum length of a non-contractible cycle in $G_n$.
By hypothesis, $G_n$ is $n$-representative.
This implies that any non-contractible cycle in $G_n$ must have length at least $n$.
This presents a contradiction ($n \leq \text{length}(P) < n$).
Therefore, no such distinct lifts $\tilde{x}, \tilde{y}$ exist.
The map $\pi|_{\tilde{H}}$ is an isomorphism, and $H$ is isomorphic to a subgraph of the planar graph $\tilde{G}_n$.
Thus, $H$ is planar.
\end{proof}

Let $G_n$ be a simple graph embedded on a surface $\Sigma$ forming an $n$-representative triangulation with vertex connectivity $\kappa(G_n) \geq n$.
The planarity of small subgraphs allows us to infer their chordality using the vertex connectivity of the host graph.

\begin{lemma}
\label{lemma:chordality}
Let $n \geq 5$. Any induced subgraph $H$ of $G_n$ with $|V(H)| < n$ is chordal.
\end{lemma}

\begin{proof}
Suppose for the sake of contradiction that $H$ is not chordal.
Then $H$ contains an induced cycle $C$ of length $\ell \geq 4$.
By Lemma \ref{lemma:planarity}, $H$ is planar, so $C$ is a Jordan curve in the local planar embedding.
Since $G_n$ is a triangulation, an induced cycle of length $\ell \geq 4$ cannot be a face boundary.
Thus, $C$ has a non trivial interior and separates the vertices in its interior from those in its exterior.
Since $H$ is an induced subgraph, the vertices of $C$ form a vertex cut in $G_n$.
The size of this cut is $|V(C)| = \ell \leq |V(H)| < n$.
This contradicts the hypothesis that $\kappa(G_n) \geq n$, as no vertex cut of size less than $n$ can exist.
Therefore, $H$ contains no induced cycles of length greater than 3.
\end{proof}

A similar argument shows that $G_n$ is $K_4$-free.

\begin{lemma}
\label{lemma:K4_free}
If $n \geq 4$, then $G_n$ is $K_4$-free.
\end{lemma}

\begin{proof}
Suppose $G_n$ contains a $K_4$ subgraph.
Since $\kappa(G_n) \geq n \geq 4$, $G_n$ has no vertex cut of size 3.
However, in any planar triangulation, every $K_4$ subgraph either constitutes the entire graph or contains a separating triangle, which is a vertex cut of size 3.
Since $|V(G_n)| > 4$, we have a contradiction.
\end{proof}

\subsection{Verification of $\psi_n$}
With these lemmas in hand, we can directly verify that $G_n$ satisfies the existence of removable vertices for small substructures.

\begin{theorem}
    \label{psi_n}
Let $n \geq 5$. Every induced subgraph $H$ of $G_n$ with $|V(H)| < n$ contains at least one removable vertex.
\end{theorem}
\begin{proof}
By Lemma \ref{lemma:chordality}, $H$ is a chordal graph.
Additionally, by Lemma \ref{lemma:K4_free}, $H$ is $K_4$-free, so from proposition \ref{finiteK4freechordal}, $H$ contains a removable vertex. Thus, $G_n\vDash \psi_n$ as desired. 

\end{proof}
\subsection{Verification of $\phi_0$}
The second axiom requires that every edge be part of exactly two triangles.
We first show that in a triangulation, every edge is incident to at least two distinct triangular faces.

\begin{lemma}
Let $G_n$ be a simple graph embedded on a surface $\Sigma$ forming an $n$-representative triangulation, with cardinality and connectivity $|V(G)|> \kappa(G_n) \geq n$, where $n \geq 3$.
For any edge $(x, y) \in E(G_n)$, there exist at least two distinct vertices $z_1, z_2 \in V(G_n)$ such that $\{x, y, z_i\}$ forms a facial triangle for $i=1, 2$.
\end{lemma}
\begin{proof}
    Let $e = (x, y)$ be an arbitrary edge in $E(G_n)$, and let $p$ be a point in its interior.
Let $\mathcal{T}$ be the set of all triangular faces in the triangulation.
Since $G_n$ is finite, $\mathcal{T}$ contains finitely many triangles.
For any triangle $T \in \mathcal{T}$ such that $e$ is not on its boundary, the distance $d(p, T)$ is strictly positive.
We define $\delta = \min \{ d(p, T) : T \in \mathcal{T}, e \not\subset \partial T \}$.
Since every face is closed and the number of faces is finite, $\delta > 0$.
Consider an open neighborhood $U = B(p, \epsilon)$ with $0 < \epsilon < \delta$.
Since $\Sigma$ is a topological manifold without boundary, $U$ is homeomorphic to an open disk in $\mathbb{R}^2$.
The interior of the edge $e$ acts as a chord that divides $U$ into exactly two disjoint open regions $U_1$ and $U_2$.
Because $\epsilon < \delta$, any face of the triangulation that intersects $U$ must contain $e$ on its boundary.
Each region $U_i$ must be contained within the interior of some face; by the definition of a triangulation, these faces $F_1, F_2$ are triangles.
Thus, $e$ is incident to at least two distinct triangles $\{x, y, z_1\}$ and $\{x, y, z_2\}$.

    We now show that $z_1 \neq z_2$.
Suppose, for the sake of contradiction, that $z_1 = z_2 = z$.
This would imply that both faces $F_1$ and $F_2$ are bounded by the same set of vertices $\{x, y, z\}$.
Consequently, the edges $(x, z)$ and $(y, z)$ would also each be incident to both $F_1$ and $F_2$.
By the same local disk argument applied to $(x, z)$ and $(y, z)$, these edges cannot be incident to any other faces.
This implies that the union $F_1 \cup F_2$ is a closed sub-surface of $\Sigma$.
Since $\Sigma$ is connected, it follows that $\Sigma = F_1 \cup F_2$, which is homeomorphic to a 2-sphere, and the graph $G_n$ is the complete graph $K_3$.

However, the hypothesis states that $|V(G)|>3$, which contradicts $G_n \cong K_3$.
Thus, $z_1$ and $z_2$ must be distinct common neighbors of $x$ and $y$.
\end{proof}

We now prove that there are no additional triangles containing any given edge, thus satisfying the axiom $\phi_0$ exactly.

\begin{theorem}
    \label{phi_0}
Let $G_n$ be a simple graph embedded on a surface $\Sigma$ that is $n$-representative with vertex connectivity $\kappa(G_n) \geq n$, where $n \geq 5$.
For any edge $(x, y) \in E(G_n)$, there exist at most two distinct vertices $z_1,z_2 \in V(G_n)$ such that $\{x, y, z_i\}$ induces a triangle.
\end{theorem}

\begin{proof}
Suppose for the sake of contradiction that there exists an edge $(x, y) \in E(G_n)$ and three distinct vertices $z_1, z_2, z_3 \in V(G_n)$ such that each $z_i$ is adjacent to both $x$ and $y$.
Let $H$ be the subgraph induced by the set of vertices $V_H = \{x, y, z_1, z_2, z_3\}$.
Note that $|V_H| = 5$.
Since $n \geq 5$, we have $|V_H| \leq n$.
By the assumption that $G_n$ is $n$-representative, any subgraph of size at most $n$ contains no non-contractible cycles and thus the restriction of the surface embedding of $G$ to a disk embedding containing $H$.  

Consider the planar embedding of $H$ given by the restriction of the surface embedding of $G_n$.
The edges incident to $x$ appear in a specific cyclic order.
We label the vertices $z_1, z_2, z_3$ such that they appear in the consecutive cyclic order $y, z_1, z_2, z_3$ around $x$.

Consider the cycle $C = (x, z_2, y, x)$, composed of the edges $E(x,z_2)$, $E(z_2,y)$, and $E(y,x)$.
By the Jordan Curve Theorem, $C$ partitions the plane into two disjoint open regions, $R_1$ and $R_2$.
Due to the cyclic ordering of edges around $x$ (specifically that $z_1$ and $z_3$ are separated by the edges $E(x,y)$ and $E(x,z_2)$), $z_1$ must lie in one region (say, $R_1$) and $z_3$ must lie in the other region ($R_2$).
Consequently, any path from $z_1$ to $z_3$ must intersect the boundary $C$.
Therefore, the set $S = \{x, y, z_2\}$ is a vertex cut that separates $z_1$ from $z_3$.

Since $|S| = 3$, the existence of this cut contradicts the assumption that $G_n$ is $5$-connected.
Thus, there cannot exist three distinct common neighbors of $x$ and $y$.
\end{proof}

\subsection{Pseudofiniteness}
Finally, we can combine these results to prove the main theorem of the paper.

\begin{theorem}[Pseudofiniteness of the Farey Graph]
    Every finite subset of $Th(F)$ is satisfied by a finite structure.
\end{theorem}
\begin{proof}
         Let $n>5$, and $G_n$ be an $n$-connected, $n$-representative triangulation of an orientable surface $\Sigma$.
By Theorems \ref{psi_n} and \ref{phi_0} above, $G_n$ satisfies the axioms $\psi_n$ and $\phi_0$ from the axiomatization given by Tent-Mohammadi in \cite{TentMohammadi2025}.
Since every finite subset of the axiomatization $\{\phi_0, \psi_1, \psi_2, \psi_3, \ldots\}$ is witnessed by some finite structure $G_n$, the theory $Th(F)$ is pseudofinite as desired.
\end{proof}
\section{Bibliography}

\printbibliography
\end{document}